\documentclass[12pt]{amsart}

\usepackage[margin=1in]{geometry}
\usepackage[all]{xy}
\usepackage[inline]{enumitem}
\usepackage{mathrsfs}
\usepackage{tikz-cd}

\def\uHom{\underline{\Hom}}
\def\coker{\operatorname{coker}}
\def\image{\operatorname{im}}
\def\coimage{\operatorname{coim}}
\def\inj{\operatorname{inj}}
\def\Bilin{\operatorname{Bilin}}
\def\Biext{\operatorname{Biext}}
\def\Hom{\operatorname{Hom}}

\def\pip{\operatorname{pip}}
\def\copip{\operatorname{copip}}

\mathchardef\mhyphen="2D
\def\mh{\mhyphen}

\newtheorem{theorem}{Theorem}
\newtheorem{proposition}[theorem]{Proposition}

\theoremstyle{remark}

\theoremstyle{definition}

\newtheorem{example}[theorem]{Example}

\begin{document}

\begin{abstract}
We describe some of the basic properties of the $2$-category of $2$-term complexes in an abelian category, using butterflies as morphisms.
\end{abstract}


\title{The $2$-category of $2$-term complexes}
\author{Jonathan Wise}
\date{\today}
\maketitle

\section{Introduction}

A Picard category, or strictly commutative $2$-group, is a category equipped with an invertible operation that is associative and commutative only up to specified isomorphisms (which themselves satisfy additional compatibility properties).  Picard categories were defined by Deligne \cite[1.4]{sga4-XVIII} in order to give an elegant description of the self-duality of the Picard group of a smooth and proper algebraic curve.  Deligne proves that Picard categories are equivalent to a $2$-category structure on $2$-term complexes of abelian groups (as well as the natural generalization for stacks of such)~\cite[Corollaire~1.4.17]{sga4-XVIII}.  However, from the point of view of complexes, the description of the $2$-category structure is somewhat indirect:  it arises either by transport of structure from the $2$-category structure on Picard categories, which requires a long list of axioms and tedious verifications; or it requires replacing complexes by complexes of injectives replacement in order to make use of the homotopy category.

The former approach does not generalize well to say, Picard stacks with action of a ring, which require an even longer list of axioms.  The latter approach does generalize (assuming enough injectives) but is still, perhaps, a little unpleasant aesthetically.  Fortunately, the $2$-category structure can be described directly using \emph{butterflies}.  I learned about these from Behrang Noohi, who used them in collaboration with Aldovandi to study $2$-groups in the noncommutative setting~\cite{Aldrovandi-Noohi-1,Aldrovandi-Noohi-2}; in the commutative setting, the structure seems to have originally been discovered by Grothendieck~\cite{sga7-VII}.  The aim of this note is to show how one can build up some of the basic theory directly with butterflies, requiring neither resolution by injectives, nor the extended system of axioms of Picard categories.

\subsection{Acknowledgements}

The author was supported by a Simon's Collaboration Grant, Award ID 636210.

\section{Butterflies}

Let $\mathscr C$ be a stack of abelian categories over a topos $X$.  A \emph{local $2$-object} of $\mathscr C$ is a $2$-term complex, concentrated in degrees $[-1,0]$.  A morphism of $2$-objects, from $E^\bullet$ to $F^\bullet$ is a \emph{butterfly}:  a commutative diagram~\eqref{eqn:33}, where the NE-SW diagonal is exact and $pj = 0$.
\begin{equation} \label{eqn:33} \vcenter{ \xymatrix{
 &&&& 0 \\
& E^{-1} \ar[dr]^j \ar[rr]^d && E \ar[ur] \\
&& Y \ar[dr]^p \ar[ur]^q \\
& F^{-1} \ar[ur]^{i} \ar[rr]^{-d} && F^0  \\
 0 \ar[ur] &&&& 
}} \end{equation}

Butterflies naturally have the structure of a groupoid, in which morphisms are morphisms on the $Y$ term, commuting with the rest of the diagram.  This gives $2$-term complexes the structure of a $2$-category in which all $2$-morphisms are invertible.  Baer sum gives this groupoid the structure of a strictly commutative $2$-group.

The \emph{identity} butterfly of $E^\bullet$ is the diagram~\eqref{eqn:33} where $Y = E^0 \oplus E^{-1}$ and $i$, $j$, $p$, and $q$ are given as follows:
\begin{align*}
i & = \begin{pmatrix} 0 \\ 1 \end{pmatrix} & j & = \begin{pmatrix} d \\ 1 \end{pmatrix}  \\
p & = \begin{pmatrix} 1 & -d  \end{pmatrix} & q & = \begin{pmatrix} 1 & 0 \end{pmatrix}
\end{align*}

We note that if $F^{-1} = 0$ then a morphism $E^\bullet \to F^\bullet$ is the same as a morphism $E^0 \to F^0$ that restricts to $0$ on $E^{-1}$.  In particular, $\mathscr C$ is fully faithfully embedded in the $2$-category of $2$-objects of $\mathscr C$ as the complexes concentrated in degree~$0$.

More generally, any morphism of complexes $\varphi^\bullet : E^\bullet \to F^\bullet$ induces a butterfly with $Y = E^0 \oplus F^{-1}$.  The diagram~\eqref{eqn:33} is given as follows:
\begin{align}
i & = \begin{pmatrix} 0 \\ 1 \end{pmatrix} & j & = \begin{pmatrix} d \\ \varphi^{-1} \end{pmatrix}  \label{eqn:4} \\
p & = \begin{pmatrix} \varphi^0 & -d \end{pmatrix} & q & = \begin{pmatrix} 1 & 0 \end{pmatrix} \notag
\end{align}
Conversely, if the SW-NE exact sequence of diagram~\eqref{eqn:33} is split, the choice of a splitting permits us to represent the butterfly $E^\bullet \to F^\bullet$ as a morphism of complexes.

\section{Composition of butterflies}
\label{sec:composition}

The composition of butterflies $Z : F^\bullet \to G^\bullet$ and $Y : E^\bullet \to F^\bullet$ is given by the homology $Z \circ Y$ of the complex~\eqref{eqn:35}:
\begin{equation} \label{eqn:35}
F^{-1} \xrightarrow{\bigl( \begin{smallmatrix} i \\ -j \end{smallmatrix} \bigr)} Y \oplus Z \xrightarrow{( \begin{smallmatrix} -p & q \end{smallmatrix} )} F^0
\end{equation}
The maps $E^{-1} \to Z \circ Y \to G^0$ are induced from $\begin{pmatrix} j \\ 0 \end{pmatrix} : E^{-1} \to Y \oplus Z$ and $\begin{pmatrix} 0 & p \end{pmatrix} : Y \oplus Z \to G^0$.  The diagonal exact sequence is induced from~\eqref{eqn:36}:
\begin{equation} \label{eqn:36}
0 \to F^{-1} \oplus G^{-1} \xrightarrow{ \bigl( \begin{smallmatrix} i & 0 \\ -j & i \end{smallmatrix} \bigr) } Y \oplus Z \xrightarrow{ \bigl( \begin{smallmatrix} q & 0 \\ -p & q \end{smallmatrix} \bigr) } E^0 \oplus F^0 \to 0
\end{equation}

Every $2$-term complex has a unique (up to unique isomorphism) butterfly to and from the zero complex.  Composition of the butterflies $E^\bullet \to 0 \to F^\bullet$ gives a butterfly~\eqref{eqn:33} with $Y = E^0 \oplus F^{-1}$ and $i$, $j$, $p$, and $q$ given as follows:
\begin{align*}
i & = \begin{pmatrix} 0 \\ 1 \end{pmatrix} & j & = \begin{pmatrix} d \\ 0 \end{pmatrix}  \\
p & = \begin{pmatrix} 0 & -d \end{pmatrix} & q & = \begin{pmatrix} 1 & 0 \end{pmatrix}
\end{align*}

Suppose we have a sequence of morphisms of butterflies:
\begin{equation}
D^\bullet \xrightarrow{X} E^\bullet \xrightarrow{Y} F^\bullet \xrightarrow{Z} G^\bullet
\end{equation}
The compositions $Z \circ (Y \circ X)$ and $(Z \circ Y) \circ X$ both arise as the homology of the following complex:
\begin{equation*}
E^{-1} \oplus F^{-1} \xrightarrow{ \Bigl( \begin{smallmatrix} i & 0 \\ -j & i \\ 0 & -j \end{smallmatrix} \Bigr)} X \oplus Y \oplus Z \xrightarrow{ \bigl( \begin{smallmatrix} -p & q & 0 \\ 0 & -p & q \end{smallmatrix} \bigr) } E^0 \oplus F^0
\end{equation*}

\section{Splitting of compositions}
\label{sec:splitting}

\begin{proposition} \label{prop:splitting}
Consider a diagram of solid arrows~\eqref{eqn:69} with exact rows, in which the central square anticommutes:
\begin{equation} \label{eqn:69} \vcenter{\xymatrix{
& 0 \ar[r] & F^{-1} \ar[r]^{i} \ar[d]_j  & Y \ar[r]^q \ar[d]^p \ar@{-->}[dl]_{-\varphi}^\varphi & E^0 \ar[r] & 0 \\
0 \ar[r] & G^{-1} \ar[r]^{i} & Z \ar[r]_q & F^0 \ar[r] & 0 
}} \end{equation}
Dashed arrows $\phi$ completing the diagram are in bijection with splittings of the exact sequence~\eqref{eqn:70}, with $Z \circ Y$ defined as the homology of the middle term of~\eqref{eqn:35}:
\begin{equation} \label{eqn:70}
0 \to G^{-1} \to Z \circ Y \to E^0 \to 0
\end{equation}
\end{proposition}
\begin{proof}
The matrix $\bigl( \begin{smallmatrix} 1 & 0 \\ \varphi & 1 \end{smallmatrix} \bigr)$ can be seen as an isomorphism between the diagrams~\eqref{eqn:44} and~\eqref{eqn:45}:
\begin{gather} \label{eqn:44} \vcenter{ \xymatrix{
&& F^{-1} \ar@{=}[r] \ar[d]_{\bigl( \begin{smallmatrix} i \\ j \end{smallmatrix} \bigr)}& F^{-1} \ar[d]^i  \\
0 \ar[r] & Z \ar[d]_q \ar[r]_-{\bigl( \begin{smallmatrix} 0 \\ 1 \end{smallmatrix} \bigr)} & Y \oplus Z \ar[r]^-{( \begin{smallmatrix} 1 & 0 \end{smallmatrix})}  \ar[d]^{( \begin{smallmatrix} p & q \end{smallmatrix} )} & Y \ar[r] & 0 \\
& F^0 \ar@{=}[r] & F^0 
}} \\ \label{eqn:45} \vcenter{ \xymatrix{
&& F^{-1} \ar@{=}[r] \ar[d]_{\bigl( \begin{smallmatrix} i \\ 0 \end{smallmatrix} \bigr)}& F^{-1} \ar[d]^i  \\
0 \ar[r] & Z \ar[d]_q \ar[r]_-{\bigl( \begin{smallmatrix} 0 \\ 1 \end{smallmatrix} \bigr)} & Y \oplus Z \ar[r]^-{( \begin{smallmatrix} 1 & 0 \end{smallmatrix})} \ar[d]^{( \begin{smallmatrix} 0 & q \end{smallmatrix} )} & Y \ar[r] & 0 \\
& F^0 \ar@{=}[r] & F^0
}} \end{gather}
Passing to homology in the vertical direction, such an isomorphism induces an isomorphism between the extensions~\eqref{eqn:46} and~\eqref{eqn:47}:
\begin{gather}
0 \to G^{-1} \to Z \circ Y \to E^0 \to 0 \label{eqn:46} \\
0 \to G^{-1} \to E^0 \oplus G^{-1} \to E^0 \to 0 \label{eqn:47}
\end{gather}
\end{proof}

If $Y : E^\bullet \to F^\bullet$ and $Z : F^\bullet \to G^\bullet$ are butterflies and $\varphi : Y \to Z$ is as in Proposition~\ref{prop:splitting} then the proposition gives an isomorphism between the butterfly $Z \circ Y$ and the butterfly $W = E^0 \oplus G^{-1} : E^\bullet \to G^\bullet$ with $i$, $j$, $p$, and $q$ defined as follows:
\begin{align*}
i & = \begin{pmatrix} 0 \\ 1 \end{pmatrix} & j & = \begin{pmatrix} d \\ i^{-1} \varphi j  \end{pmatrix} \\
p & = \begin{pmatrix} -p \varphi q^{-1} & -d \end{pmatrix} & q & = \begin{pmatrix} 1 & 0 \end{pmatrix}
\end{align*}

That is, $Z \circ Y$ is butterfly arising from the morphism  of complexes $\psi : E^\bullet \to F^\bullet$ given by the following formulas:
\begin{equation*}
\psi^{-1} = i^{-1} \varphi j \qquad \psi^0 = -p \varphi q^{-1}
\end{equation*}

\section{Composition with morphisms of complexes}
\label{sec:comp-with-morph}

Suppose that $Y : E^\bullet \to F^\bullet$ and $Z : F^\bullet \to G^\bullet$ are butterflies, with $Y$ representable by a morphism of complexes.  We choose an isomorphism $Y \simeq E^0 \oplus F^{-1}$, so that the butterfly has the formula~\eqref{eqn:4}.  With this identification, the composition $Z \circ Y$ is given by the homology of~\eqref{eqn:5}:
\begin{equation} \label{eqn:5}
F^{-1} \xrightarrow{\Bigl(\begin{smallmatrix} 0 \\ 1 \\ -j \end{smallmatrix}\Bigr)} E^0 \oplus F^{-1} \oplus Z \xrightarrow{( \begin{smallmatrix} -\varphi^0 & d & q \end{smallmatrix} )} F^0
\end{equation}
This reduces to the kernel of~\eqref{eqn:6}:
\begin{equation} \label{eqn:6}
E^0 \oplus Z \xrightarrow{ ( \begin{smallmatrix} - \varphi^0 & q \end{smallmatrix} ) } F^0
\end{equation}
In other words, $Z \circ Y = {\varphi^0}^\ast(Z)$ with $j_{Z \circ Y} = \begin{pmatrix} d \\ j \varphi^{-1} \end{pmatrix}$ and $q_{Z \circ Y} = \begin{pmatrix} 0 & q \end{pmatrix}$. 

We obtain a similar formula if it is $Z$ instead of $Y$ that is represented by a morphism of complexes $\varphi$.  In this case, $Z \circ Y$ is given by the homology of~\eqref{eqn:7}, which reduces to the cokernel of~\eqref{eqn:8}:
\begin{gather}
F^{-1} \xrightarrow{ \Bigl( \begin{smallmatrix} i \\ -d \\ -\varphi^{-1} \end{smallmatrix} \Bigr) } Y \oplus F^0 \oplus G^{-1} \xrightarrow{ ( \begin{smallmatrix} -p & 1 & 0 \end{smallmatrix} ) } F^0 \label{eqn:7} \\
F^{-1} \xrightarrow{ \bigl( \begin{smallmatrix} i \\ - \varphi^{-1} \end{smallmatrix} \bigr) } Y \oplus G^{-1} \label{eqn:8}
\end{gather}
Thus $Z \circ Y = \varphi^{-1}_\ast(Y)$ with $j_{Z \circ Y} = \begin{pmatrix} j \\ 0 \end{pmatrix}$ and $p_{Z \circ Y} = \begin{pmatrix} \varphi^0 p & -d \end{pmatrix}$.

\section{Homology}
\label{sec:homology}

Suppose that $Y : E^\bullet \to F^\bullet$ is a butterfly.  By a straightforward diagram chase, $Y$ induces morphisms $Y_\ast : H^i(E^\bullet) \to H^i(F^\bullet)$ for $i = -1, 0$.  On $H^{-1}$, we can write $Y_\ast = i^{-1} j$ and on $H^0$ we can write $Z_\ast = p q^{-1}$.  

\begin{proposition} \label{prop:homology-functor}
If $Y : E^\bullet \to F^\bullet$ and $Z : F^\bullet \to G^\bullet$ are butterflies then $(Z \circ Y)_\ast = Z_\ast \circ Y_\ast$ as maps $H^i(E^\bullet) \to H^i(G^\bullet)$.
\end{proposition}
\begin{proof}
Observe first that $\begin{pmatrix} i \\ 0 \end{pmatrix} : Y \to Y \oplus Z$ and $\begin{pmatrix} 0 \\ j \end{pmatrix}$ both induce the same map $H^{-1} F^\bullet \to Z \circ Y$.  Likewise $\begin{pmatrix} p & 0 \end{pmatrix} : Y \oplus Z \to F^0$ and $\begin{pmatrix} 0 & q \end{pmatrix} : Y \oplus Z \to F^0$ both induce the same map $Z \oplus Y \to H^0 F^\bullet$.  Using these maps, we illustrate the composition on homology with one diagram for $H^0$ and another for $H^1$:
\begin{equation*} \begin{tikzcd}[row sep=large]
H^{-1} E^\bullet \ar[r,hook] \ar[d,"Y_\ast",swap] & E^{-1} \ar[r,"j_Y"] \ar[d,"j_{Z\circ Y}" near end] & Y \ar[d,"\bigl( \begin{smallmatrix} 1 \\ 0 \end{smallmatrix} \bigr)"] \\
H^{-1} F^\bullet \ar[r,"r"] \ar[d,"Z_\ast",swap] \ar[urr,"i_Y" near start] \ar[drr,"j_Z" near start, swap] & Z \circ Y \ar[r,hook]  & \displaystyle Y \mathop{\oplus}^{F^{-1}} Z \\
H^{-1} G^\bullet \ar[r,hook] & G^{-1} \ar[r,"i_Z",swap] \ar[u,"i_{Z\circ Y}" near end, swap] & Z \ar[u,"\bigl( \begin{smallmatrix} 0 \\ 1 \end{smallmatrix} \bigr)", swap]
\end{tikzcd} \qquad
\begin{tikzcd}[row sep=large,ampersand replacement=\&]
Y \ar[r,"q_Y"] \ar[drr,"p_Y" near start,swap] \& E^0 \ar[r,two heads] \& H^0 E^\bullet \ar[d,"Y_\ast"] \\
\displaystyle Y \mathop\oplus_{F^0} Z \ar[u,"{( \begin{smallmatrix} 0 & 1 \end{smallmatrix} )}"] \ar[d,"{( \begin{smallmatrix} 1 & 0 \end{smallmatrix} )}", swap] \arrow[r, two heads] \& Z \circ Y \ar[r,"s"] \ar[u,"q_{Z\circ Y}" near end, swap] \ar[d,"p_{Z \circ Y}", near end] \& H^0 F^\bullet \ar[d,"Z_\ast"] \\
Z \ar[r] \ar[urr,"q_Z",near start] \& G^0 \ar[r,two heads] \& H^0 G^\bullet
\end{tikzcd} \end{equation*}
If $\alpha \in H^{-1} E^\bullet$ then its image $Y_\ast(\alpha) \in H^{-1} F^\bullet$ is the unique $\beta$ such that $i_Y(\beta) = j_Y(\alpha)$.  By commutativity of the diagram, $r(\beta) = j_{Z\circ Y}(\alpha)$.  Similarly, $r(\beta) = i_{Z\circ Y}(\gamma)$.  As $(Z \circ Y)_\ast (\alpha)$ is the unique $\gamma$ such that $j_{Z\circ Y}(\alpha) = i_{Z\circ Y}(\gamma)$, we obtain $(Z\circ Y)_\ast (\alpha) = Z_\ast Y_\ast(\alpha)$ on $H^{-1}$, as required.  The same argument in the second diagram, applied with arrows reversed, gives the conclusion for $H^0$.
\end{proof}

\section{Invertible morphisms}

\begin{proposition}
The following properties are equivalent of $Y : E^\bullet \to F^\bullet$:
\begin{enumerate}[label=(\roman{*})]
\item \label{it:iso} the butterfly $Y$ has a $2$-sided inverse (up to isomorphism).
\item \label{it:qis} the maps $H^i(E^\bullet) \to H^i(F^\bullet)$ are isomorphisms for $i = -1, 0$;
\item \label{it:cone} the sequence~\eqref{eqn:48} is exact:
\begin{equation} \label{eqn:48}
0 \to E^{-1} \to Y \to F^0 \to 0
\end{equation}
\end{enumerate}
\end{proposition}
\begin{proof}
Certainly an isomorphism will induce isomorphisms on homology, by functoriality (Proposition~\ref{prop:homology-functor}), so \ref{it:iso} implies \ref{it:qis}.  That \ref{it:qis} implies \ref{it:cone} can be verified by a diagram chase.  Finally, if~\eqref{eqn:48} is exact then the inverse $Y'$ of $Y$ is obtained by reflecting $Y$ across the horizontal and inverting $i$ and $q$.  The identity map $Y \to Y$ therefore gives a splitting as in Proposition~\ref{prop:splitting}, and we obtain $j = \bigl( \begin{smallmatrix} d \\ 1 \end{smallmatrix} \bigr)$ and $p = \begin{pmatrix} 1 & -d \end{pmatrix}$.  The compositions $Y' \circ Y$ and $Y \circ Y'$ are therefore identity maps.
\end{proof}

\section{Kernels and cokernels}
\label{sec:kernels}

The kernel of~\eqref{eqn:33} is the complex~\eqref{eqn:39}:
\begin{equation} \label{eqn:39}
\ker(Y) = [ E^{-1} \xrightarrow j \ker(p) ]
\end{equation}
The morphism of complexes given by the restriction of $q$ induces a map $\ker(Y) \to E^\bullet$.  We note that $\ker(Y)$ has the universal property of a kernel:

\begin{proposition}
Let $Z : F^\bullet \to G^\bullet$ be a butterfly.  For any complex $E^\bullet$, there an equivalence of groupoids between the butterflies $E^\bullet \to \ker(Z)$ and the pairs $(Y, \psi)$ where $Y : E^\bullet \to F^\bullet$ is a butterfly and $\psi : Z \circ Y \simeq 0$ is an isomorphism of butterflies.
\end{proposition}
\begin{proof}
The groupoid of butterflies $Y : E^\bullet \to \ker(Z)$ is equivalent to the groupoid of diagrams~\eqref{eqn:40} with commutative triangles (and therefore with anticommutative central rhombus), having exact NE-SW diagonals, and whose NW-SE diagonals compose to zero:
\begin{equation} \label{eqn:40} \vcenter{ \xymatrix@C=4em{
&&&& 0  \\
& E^{-1} \ar[rr]^d \ar[dr]^j \ar[dddr]_(.3)0|!{[dd];[dr]}\hole|!{[dd];[ddrr]}\hole && E^0 \ar[ur] \\
&& Y \ar[ur]^q \ar[dr]^p \ar@{-->}[dd]_(.25){-\varphi}^(.75)\varphi \ar[dddr]^(.7)0|!{[dl];[dr]}\hole|!{[dd];[dr]}\hole && 0\\
& F^{-1} \ar[ur]^(.65){i} \ar[dr]_j \ar[rr]|-\hole^(.4){-d}_(.6)d && F^0 \ar[ur] \\
0 \ar[ur] && Z \ar[ur]_(.65)q \ar[dr]^p \\
& G^{-1} \ar[ur]^{i} \ar[rr]^{-d} && G^0 \\
0 \ar[ur]
}} \end{equation}

By Proposition~\ref{prop:splitting} and the discussion following it, such diagrams are equivalent to isomorphisms between $Z \circ Y$ and the zero homomorphism $E^\bullet \to G^\bullet$.
\end{proof}

By passing to the opposite category, the proposition shows the existence of cokernels in $2\mh\mathscr C$.  The cokernel of~\eqref{eqn:33} is the complex~\eqref{eqn:1}, with differential induced from $p$, and the map from $F^\bullet$ induced from $i$:
\begin{equation} \label{eqn:1}
\coker(Y) = [ \coker(j) \xrightarrow{-p} F^0 ]
\end{equation}

\section{Monomorphisms and epimorphisms}

We call a butterfly a \emph{monomorphism} if $Y \circ (-)$ is a fully faithful functor.  Similarly, we call $Y$ an \emph{epimorphism} if $(-) \circ Y$ is fully faithful.  Since butterflies have an additive structure, a morphism is a monomorphism if and only if its kernel is zero.  Likewise, it is an epimorphism if and only if its cokernel is zero.

\begin{proposition}
Let $Y$ be a butterfly~\eqref{eqn:33}.  The following are equivalent:
\begin{enumerate}[label=(\roman{*})]
\item \label{it:mono} $Y$ is a monomorphism;
\item \label{it:ker-0} $\ker(Y) = 0$;
\item \label{it:left-ex} the sequence~\eqref{eqn:54} is exact:
\begin{equation} \label{eqn:54}
0 \to E^{-1} \xrightarrow{j} Y \xrightarrow{p} F^0 
\end{equation}
\end{enumerate}
\end{proposition}
\begin{proof}
The equivalence of~\ref{it:mono} and~\ref{it:ker-0} is immediate from the additive structure on morphisms.  The equivalence of~\ref{it:ker-0} and~\ref{it:left-ex} comes from the explicit formula for kernels computed in Section~\ref{sec:kernels}.
\end{proof}

Reversing arrows gives the dual version.

\section{Pips and copips}
\label{sec:pips}

If $Y : E^\bullet \to F^\bullet$ is butterfly, the pip of $Y$ is $\ker(j) \subset E^{-1}$, positioned in degree $-1$.  Similarly, the copip is $\coker(p)$, positioned in degree~$0$.  The terminology is borrowed from~\cite{dupont}.

We will say that $Y$ is faithful if the functor $Y_\ast : \Hom(D^\bullet, E^\bullet) \to \Hom(D^\bullet, F^\bullet)$ is faithful for all $2$-objects $D^\bullet$.  Likewise, we say $Y$ is cofaithful if $Y^\ast : \Hom(F^\bullet, G^\bullet) \to \Hom(E^\bullet, G^\bullet)$ is faithful for all $2$-objects $G^\bullet$.

\begin{proposition}
Let $Y : E^\bullet \to F^\bullet$ be a butteryfly.  The following are equivalent:
\begin{enumerate}[label=(\roman{*})]
\item \label{it:faithful} $Y$ is faithful;
\item \label{it:pip0} $\pip(Y) = 0$;
\item \label{it:j-inj} $E^{-1} \to Y$ is injective;
\item \label{it:h1-inj} $H^{-1}(E^\bullet) \to H^{-1}(F^\bullet)$ is injective.
\end{enumerate}
\end{proposition}
\begin{proof}
The equivalence of \ref{it:pip0} and \ref{it:j-inj} is the definition.  Note that by the additive structure of $\Hom(D^\bullet, E^\bullet)$, faithfullness is equivalent to injectivity on the automorphism group of the zero object.  This group works out to the subgroup of $\varphi : D^0 \to E^{-1}$ such that $\varphi d = d \varphi = 0$.  In other words, it is the group of homomorphisms $H^0(D^\bullet) \to H^{-1}(E^\bullet)$.  This gives the equivalence of \ref{it:faithful} and \ref{it:h1-inj}.  Finally, a diagram chase shows that $\pip(Y)$ is the kernel of $H^{-1}(E^\bullet) \to H^{-1}(F^\bullet)$.
\end{proof}

\begin{proposition}
Suppose that $Y : E^\bullet \to F^\bullet$ is a butterfly.  There are canonical isomorphisms:
\begin{gather*}
\coker(\ker(Y)) \xrightarrow{\sim} \ker(\copip(Y)) \\
\coker(\pip(Y)) \xrightarrow{\sim} \ker(\coker(Y))
\end{gather*}
\end{proposition}
\begin{proof}
The maps in question are given by the following butterflies:
\begin{equation*} \xymatrix@!C=15pt{
0 \ar[dr] &&&& 0 \\
& \ker(p) \ar[dr] \ar[rr] && E^0 \ar[ur] \\
&&  Y \ar[ur] \ar[dr] \\
&F^{-1} \ar[ur] \ar[rr] && \image(p) \ar[dr] \\
0 \ar[ur]  &&&& 0
} \hskip4em \xymatrix@!C=15pt{
0 \ar[dr] &&&& 0 \\
& \coimage(j) \ar[dr] \ar[rr] && E^0 \ar[ur] \\
& & Y \ar[ur] \ar[dr] \\
& F^{-1} \ar[ur] \ar[rr] && \coker(j) \ar[dr] \\
0 \ar[ur] &&&& 0
} \end{equation*}
Since the diagonals are exact, these are isomorphisms.
\end{proof}

\section{Images and coimages}
\label{sec:images}

We define the image and coimage:
\begin{gather*}
\image(Y) = \ker(\copip(Y)) \\
\coimage(Y) = \coker(\pip(Y))
\end{gather*}
Note that the image and coimage are not isomorphic in general.  Note also that there is generally no map from $\coker(\ker(Y))$ to $\ker(\coker(Y))$.

\begin{proposition}
Let $Y : F^\bullet \to G^\bullet$ be a butterfly and suppose that the sequence~\eqref{eqn:9} is exact:
\begin{equation} \label{eqn:9}
E^{-1} \xrightarrow j Y \xrightarrow p F^0
\end{equation}
Then there are canonical isomorphisms:
\begin{equation} \label{eqn:18}
\coker(\pip(Y)) \xrightarrow{\sim}  \coker(\ker(Y)) \xrightarrow{\sim} \ker(\coker(Y)) \xrightarrow{\sim} \ker(\copip(Y)) 
\end{equation}
\end{proposition}
\begin{proof} 
If $Y : E^\bullet \to F^\bullet$ is a butterfly, then its coimage is given by $[ \ker(p) \to E^0 ]$; its image is given by $[ F^{-1} \to \coker(j) ]$.  The isomorphisms above are therefore represented by the following diagram:
\begin{equation} \label{eqn:43} \vcenter{ \xymatrix@!R=10pt{
& \coimage(j) \ar[dd] \ar[rr] & & E^0 \ar[dd] \\
0 \ar[dr] &&&& 0 \\
& \ker(p) \ar[rr] \ar[dr] && E^0 \ar[ur] \\
&& Y \ar[ur] \ar[dr] \\
& F^{-1} \ar[rr] \ar[dd] \ar[ur] & & \coker(j) \ar[dd] \ar[dr] \\
0 \ar[ur] &&&& 0 \\
& F^{-1} \ar[rr] && \image(p)
}} \end{equation}
Since both diagonals are exact, this morphism is invertible.
\end{proof}

\section{Exact sequences}

A sequence of butterflies~\eqref{eqn:49} is called a \emph{complex} if
\begin{equation} \label{eqn:49}
D^\bullet \xrightarrow{d} E^\bullet \xrightarrow{d} F^\bullet \xrightarrow{d} G^\bullet
\end{equation}
if it is equipped with isomorphisms $\varphi : d^2 \simeq 0$ such that $d \varphi = \varphi d = 0$.  More explicitly, the exact sequence~\eqref{eqn:49} is a diagram~\eqref{eqn:71} in which triangles are commutative (so rhombi are anticommutative), NE-SW diagonals are exact, NW-SE diagonals compose to zero, and $\psi \varphi = 0$:
\begin{equation} \label{eqn:71} \vcenter{\xymatrix@C=4em{
&&&& 0 \\
& D^{-1} \ar[rr] \ar[dr]^j \ar[dddr]_(.3)0|!{[dd];[dr]}\hole|!{[dd];[ddrr]}\hole && D^0 \ar[ur] \\
&& X \ar[dd]_(.25){-\varphi}^(.75)\varphi \ar[ur] \ar[dr] \ar[dddr]^(.7)0|!{[dl];[dr]}\hole|!{[dd];[dr]}\hole && 0 \\
& E^{-1} \ar[rr]|-\hole^(.4){-d}_(.6)d \ar[ur] \ar[dddr]_(.3)0|!{[dd];[dr]}\hole|!{[dd];[ddrr]}\hole \ar[dr] && E^0 \ar[ur] \\
0 \ar[ur] && Y \ar[dd]_(.25){-\psi}^(.75)\psi \ar[ur] \ar[dr] \ar[dddr]^(.7)0|!{[dl];[dr]}\hole|!{[dd];[dr]}\hole && 0 \\
& F^{-1} \ar[rr]|-\hole^(.4){-d}_(.6)d \ar[ur] \ar[dr]  && F^0 \ar[ur] \\
0 \ar[ur] && Z \ar[ur]  \ar[dr]_p \\
& G^{-1} \ar[rr] \ar[ur] && G^0 \\
0 \ar[ur]
}} \end{equation}
It is called exact if $d$, $\varphi$, and $\psi$ induce an isomorphism $\coker(D^\bullet \to E^\bullet) \simeq \ker(F^\bullet \to G^\bullet)$.  In other words, the sequence~\eqref{eqn:72} should be exact:
\begin{equation} \label{eqn:72}
0 \to \coker(j) \xrightarrow{\varphi} Y \xrightarrow{\psi} \ker(p) \to 0
\end{equation}

\begin{proposition}
Suppose that~\eqref{eqn:53} is a complex of butterflies:
\begin{equation} \label{eqn:53}
0 \to E^\bullet \xrightarrow{Y} F^\bullet \xrightarrow{Z} G^\bullet \to 0
\end{equation}
\begin{enumerate}[label=(\roman{*})]
\item It is left exact if and only if the sequence~\eqref{eqn:12} is exact:
\begin{equation}
0 \to E^{-1} \xrightarrow{j} Y \xrightarrow{\varphi} Z \xrightarrow{p} G^0 \label{eqn:12} 
\end{equation}
\item It is right exact if and only if the sequence~\eqref{eqn:13} is exact:
\begin{equation}
E^{-1} \xrightarrow{j} Y \xrightarrow{\varphi} Z \xrightarrow{p} G^0 \to 0 \label{eqn:13} 
\end{equation}
\item It is exact if and only if~\eqref{eqn:73} is exact:
\begin{equation}
0 \to E^{-1} \xrightarrow{j} Y \xrightarrow{\varphi} Z \xrightarrow{p} G^0 \to 0 \label{eqn:73}
\end{equation}
\end{enumerate}
\end{proposition}
\begin{proof}
Left and right exactness of~\eqref{eqn:53} correspond to the exactness of the following two sequences:
\begin{gather}
0 \to E^\bullet \xrightarrow{Y} F^\bullet \xrightarrow{Z} G^\bullet \label{eqn:76} \\
E^\bullet \xrightarrow{Y} F^\bullet \xrightarrow{Z} G^\bullet \to 0 \label{eqn:77}
\end{gather}
In other words, $E^\bullet \to \ker(Z)$ and $\coker(Y) \to G^\bullet$ should be isomorphisms.  These properties correspond, respectively, to the exactness of the following two sequences:
\begin{gather}
0 \to E^{-1} \to Y \to \ker(p) \to 0 \label{eqn:74} \\
0 \to \coker(j) \to Z \to G^0 \to 0 \label{eqn:75}
\end{gather}
Exactness of each of these sequences is, respectively, equivalent to the exactness of the following two:
\begin{gather*}
0 \to E^{-1} \to Y \to Z \to G^0 \\
E^{-1} \to Y \to Z \to G^0 \to 0
\end{gather*}
These are both exact if and only if~\eqref{eqn:73} is exact, if and only if~\eqref{eqn:74} is exact and $Z \to G^0$ is surjective, if and only if~\eqref{eqn:75} is exact and $E^{-1} \to Y$ is injective.
\end{proof}

\begin{example}
The sequences~\eqref{eqn:51},~\eqref{eqn:10}, and~\eqref{eqn:52} are all exact:
\begin{gather}
0 \to E^{-1} \to E^0 \to E^\bullet \to 0 \label{eqn:51} \\
0 \to E^0 \to E^\bullet \to E^{-1}[1] \to 0 \label{eqn:10} \\
0 \to E^\bullet \to E^{-1}[1] \to E^0[1] \to 0 \label{eqn:52}
\end{gather} 
\end{example}

\begin{proposition}
Suppose that~\eqref{eqn:53} is \begin{enumerate*}[label=(\roman{*})] \item left exact, \item right exact, or \item exact \end{enumerate*}.  Then there is an induced exact sequence~\eqref{eqn:78}:
\begin{equation} \label{eqn:78}
0 \to H^{-1} E^\bullet \to H^{-1} F^\bullet \to H^{-1} G^\bullet \to H^0 E^\bullet \to H^0 F^\bullet \to H^0 G^\bullet \to 0
\end{equation}
\end{proposition}
\begin{proof}
We construct the arrow $H^{-1} G^\bullet \to H^0 E^\bullet$ and show the exactness at the $H^{-1} G^\bullet$ and $H^0 E^\bullet$ positions.  Divide diagram~\eqref{eqn:40} by $E^{-1}$  and take kernels into $G^0$.  This preserves exactness and produces the following diagram:
\begin{equation*} \xymatrix{
&&&& 0 \\
& 0 \ar[rr] & & H^0 E^\bullet \ar[ur] \\
&& Y' \ar[dd]_{-\varphi'}^{\varphi'} \ar[ur] \ar[dr] \\
& F^{-1} \ar[ur] \ar[dr] && F^0 \\
&& Z' \ar[ur] \ar[dr] \\
& H^{-1} G^\bullet \ar[ur] \ar[rr] && 0 \\
0 \ar[ur]
} \end{equation*}
The map $\varphi' : Y' \to Z'$ is an isomorphism because it is induced from the exact sequence~\eqref{eqn:73} by dividing by $E^{-1}$ and taking kernels into $G^0$.  Therefore we obtain a map $H^{-1} G^\bullet \to H^0 E^\bullet$.  It is easy to see from a diagram chase that the kernel of this map is the image of $H^{-1} F^\bullet \to H^{-1} G^\bullet$ and (dually) that the cokernel injects into $H^0 F^\bullet$.  

If we instead take kernels into $E^0$, $F^0$, and $G^0$ in the exact sequence~\eqref{eqn:12}, we get the exactness of~\eqref{eqn:14}, and if we divide by $E^{-1}$, $F^{-1}$, and $G^{-1}$ in~\eqref{eqn:13}, we get the exactness of~\eqref{eqn:15}:
\begin{gather}
0 \to H^{-1} E^\bullet \to H^{-1} F^\bullet \to H^{-1} G^\bullet \label{eqn:14} \\
H^0 E^\bullet \to H^0 F^\bullet \to H^0 G^\bullet \to 0 \label{eqn:15}
\end{gather}
\end{proof}

\section{Derived functors}

Assume that $\mathscr C$ has enough injectives.  Then every object of $2\mh\mathscr C$ is isomorphism to an object $E^\bullet$ with $E^{-1}$ injective.  We denote the full subcategory of such objects $\inj(2\mh\mathscr C)$.

Let $\Phi : \mathscr C \to \mathscr D$ be a left exact functor.  We extend $\Phi$ to $\inj(2\mh\mathscr C)$ by applying it objectwise.  If $Y : E^\bullet \to F^\bullet$ is a butterfly, with $F^\bullet \in \inj(2\mh\mathscr C)$ then the sequence~\eqref{eqn:55} is split exact, so~\eqref{eqn:56} is exact as well:
\begin{gather} 
0 \to F^{-1} \to Y \to E^0 \to 0 \label{eqn:55} \\
0 \to \Phi F^{-1} \to \Phi Y \to \Phi E^0 \to 0 \label{eqn:56}
\end{gather}
This means that $\Phi Y$ is a butterfly from $\Phi E^\bullet$ to $\Phi F^\bullet$.  By the equivalence of $2\mh\mathscr C$ with $\inj(2\mh\mathscr C)$ this extends $\Phi$:
\begin{equation*}
\mathrm R^{[0,1]} \Phi : 2\mh\mathscr C \simeq \inj(2\mh\mathscr C) \to 2\mh\mathscr D
\end{equation*}
One may of course define $\mathrm R^{[0,1]} \Phi$ if $\mathscr C$ has enough $\Phi$-acyclic objects.

If $\Phi$ is right exact and $\mathscr C$ has enough projectives, the left derived functors of $\Phi$ can be defined by the same procedure with arrows reversed.

\begin{proposition}
If $\Phi$ is left exact then $\mathrm R^{[0,1]}\Phi$ is also left exact.
\end{proposition}
\begin{proof}
This is immediate from the explicit formula for the kernel.
\end{proof}

\section{Biextensions}

Assume that $\mathscr C$ is a closed symmetric monoidal category, meaning we have adjoint functors:
\begin{gather*}
\otimes : \mathscr C \times \mathscr C \to \mathscr C : (A, B) \mapsto A \otimes B \\
\uHom : \mathscr C^{\rm op} \times \mathscr C \to \mathscr C : (A, B) \mapsto \uHom(A,B)
\end{gather*}
These combine to a tri-functor:
\begin{gather*}
\Bilin : \mathscr C^{\rm op} \times \mathscr C^{\rm op} \times \mathscr C \to \mathbf{Sets} \\
\Bilin(A,B;C) = \Hom(A \otimes B, C) = \Hom(A, \uHom(B, C)) = \Hom(B, \uHom(A,C))
\end{gather*}
We assume that $\otimes$ and $\uHom$ extend to right- and left-exact functors on $2\mh\mathscr C$, respectively.  Then $\Bilin$ also extends to $2\mh\mathscr C$, valued in groupoids.  Combined with the inclusions $\mathscr C \to 2\mh\mathscr C$, we obtain a functor:
\begin{gather*}
\Biext : \mathscr C^{\rm op} \times \mathscr C^{\rm op} \times \mathscr C \to \mathbf{Gpds} \\
\Biext(A,B;C) = \Bilin(A,B;C[1]) = \Hom(A \mathop{\otimes}^{\mathrm L_{[0,1]}} B, C[1]) = \Hom(A, \mathrm R^{[0,1]} \uHom(B, C))
\end{gather*}

When explicit resolutions are available, one can be more explicit about the meaning of a biextension.  For example, suppose that $\mathscr C$ is the category of sheaves of abelian groups on a site $X$.  For each $A \in \mathscr C$, let $\mathbf Z A$ be the sheaf of abelian groups generated by the underlying sheaf of $A$.  We write $\mathbf Z^n A = \mathbf Z(A \times \cdots \times A)$.  Then we have an exact sequence:
\begin{equation*}
\cdots \to \mathbf Z^3 A \to \mathbf Z^2 A \to \mathbf Z A \to A \to 0
\end{equation*}
Resolving $B$ similarly, we have a resolution of $A \mathop{\otimes}^{\mathrm L_{[0,1]}} B$:
\begin{equation*} 
(\mathbf Z^3 A \otimes \mathbf Z B) \oplus ( \mathbf Z^2 A \otimes \mathbf Z^2 B) \oplus (\mathbf Z A \otimes \mathbf Z^3 B) \to (\mathbf Z^2 A \otimes \mathbf Z B) \oplus (\mathbf Z A \otimes \mathbf Z^2 B) \to \mathbf Z A \otimes \mathbf Z B \to A \mathop{\otimes}^{\mathrm L_{[0,1]}} B \to 0
\end{equation*}
A biextension of $A$ and $B$ by $C$ is an extension of $\mathbf Z A \otimes \mathbf Z B$ by $C$ together with a trivialization of that biextension over $(\mathbf Z^2 A \otimes \mathbf Z B) \oplus (\mathbf Z A \otimes \mathbf Z^2 B)$ that restricts to the trivial trivialization over the next term in the sequence.  An extension of $\mathbf Z A \otimes \mathbf Z B = \mathbf Z(A \times B)$ by $C$ is a $C$-torsor $P$ over $A \times B$.  The trivialization over $\mathbf Z^2 A \oplus \mathbf Z B$, together with the trivialization of this trivialization over $\mathbf Z^3 A \otimes \mathbf Z B$, gives the fiber of $P$ over $A \times \{ b \}$ the structure of an extension of $A$ by $C$.  Similarly, the fiber over $\{ a \} \times B$ has the structure of an extension of $B$ by $C$.  Finally, the trivialization over $\mathbf Z^2 A \oplus \mathbf Z^2 B$ asserts that the two addition laws on $P$ commute with one another, where it makes sense:  if $x \in P(a,b)$, $x' \in P(a',b)$, $y \in P(a,b')$ and $y' \in P(a',b')$, then we have
\begin{equation*}
(x \mathop+_{P(A,b)} x') \mathop+_{P(a+a',B)} (y \mathop+_{P(A,b')} y') = (x \mathop+_{P(a,B)} y) \mathop+_{P(A,b+b')} (x' \mathop+_{P(a',B)} y')
\end{equation*}
This recovers the definition of a biextension from \cite[D\'efinition~2.1]{sga7-VII}.

\section{Globalization}

The $2$-objects of $\mathscr C$, with butterflies as morphisms, form a fibered $2$-category over $X$.  For any two fixed $E^\bullet$ and $F^\bullet$ in $\mathscr C(U)$, the morphisms from $E^\bullet$ to $F^\bullet$ form a stack in groupoids.  In many situations, the local $2$-objects of $\mathscr C$ form a stack, but in general we define $2\mhyphen\mathscr C$ to be the associated $2$-stack of this fibered category.  

\begin{proposition}
Let $X$ be a site and let $\mathscr C \to X^{\rm op}$ be a bifibered category with abelian fibers that forms a stack over $X$.  If the fibers of $\mathscr C$ have enough injectives and admit all products then $2\mh\mathscr C$ is a $2$-stack over $X$.
\end{proposition}
\begin{proof}
Since butterflies from $E^\bullet$ to $F^\bullet$ form a stack, the stackification of $2\mh\mathscr C$ is the the category of descent data for $2\mh\mathscr C$.  Let $R$ be a covering sieve (of $X$, without loss of generality) and let $E^\bullet$ be a descent datum for $2\mh\mathscr C$ over $R$.  We wish to show that $E^\bullet$ is isomorphic to an object of $2\mh\mathscr C$.

Since $\mathscr C(U)$ has enough injectives for all $U$, we can assume that $E^\bullet(U)$ is represented by a complex $E^{-1} \to E^0$ with $E^{-1}$ injective.  This gives a morphism:
\begin{equation*}
j^\ast E^\bullet \simeq E^\bullet(U) \to E^{-1}
\end{equation*}
By adjunction, we obtain a map $E^\bullet \to j_\ast E^{-1}[1]$.  We define $F^{-1} = \prod_{j : U \to X} j_\ast E^{-1}$.  We obtain $E^\bullet \to F^{-1}[1]$.  Let $Q^\bullet$ be the cokernel of $E^\bullet \to F^{-1}[1]$.  A priori, this is another descent datum for $2\mh\mathscr C$, but by the long exact sequence we have $H^0 Q^\bullet = 0$.  Thus $Q^\bullet$ is isomorphic to an object $F^0$ in $\mathscr C[1]$, and we have an exact sequence:
\begin{equation*}
0 \to E^\bullet \to F^{-1} \to F^0 \to 0
\end{equation*}
That is, $E^\bullet \simeq F^\bullet$, as required.
\end{proof}

\bibliographystyle{amsalpha}
\bibliography{2ab}

\end{document}